\documentclass[12pt]{amsart}
\usepackage{amssymb}
\usepackage{amsmath}
\usepackage{amsthm}

\newcommand{\A}{{\mathcal A}}

\newcommand{\C}{\ensuremath{\mathbb{C}}}
\newcommand{\R}{\ensuremath{\mathbb{R}}}

\newcommand{\p}{\partial}

\newcommand{\E}{{\mathcal E}}

\newcommand{\Z}{\ensuremath{\mathbb{Z}}}
\newtheorem{lemma}{Lemma}

\newtheorem{theorem}{Theorem}

\begin{document}
\title[Star products admitting a smooth extension]
{Star products with separation of variables admitting a smooth extension}
\author[Alexander Karabegov]{Alexander Karabegov}
\address[Alexander Karabegov]{Department of Mathematics, Abilene Christian University, ACU Box 28012, Abilene, TX 79699-8012}
\email{axk02d@acu.edu}

\begin{abstract}
Given a complex manifold $M$ with an open dense subset $\Omega$ endowed with a pseudo-K\"ahler form $\omega$ which cannot be smoothly extended to a larger open subset, we consider various examples where the corresponding K\"ahler-Poisson structure and a star product with separation of variables on $(\Omega, \omega)$ admit smooth extensions to $M$. We suggest a simple criterion of the existence of a smooth extension of a star product and apply it to these examples.
\end{abstract}

\date{June 14, 2010}
\maketitle
\section{Introduction}
A formal differential star product on a Poisson manifold $(M,\{\cdot,\cdot\})$ is an associative product on the space $C^\infty(M)[[\nu]]$ of smooth complex-valued formal functions on $M$ given by the formula
\begin{equation}\label{E:star}
 f\ast g = \sum_{r \geq 0} \nu^r C_r(f,g),
\end{equation}
where $C_r$ are bidifferential operators on $M$, $C_0(f,g) = fg$ and $C_1(f,g) -C_1(g,f) = i\{f,g\}$ (see \cite{BFFLS}). It was proved by Kontsevich in \cite{K} that deformation quantizations exist on any Poisson manifold.

We will assume that the unit constant is the unity with respect to the star product: $f\ast 1 = 1\ast f = f$ for all $f \in C^\infty(M)[[\nu]]$. Given functions $f,g \in C^\infty(M)[[\nu]]$, we will denote by $L_f$ and $R_g$ the left star multiplication operator by $f$ and the right star multiplication operator by $g$, respectively, so that $f\ast g = L_f g = R_g f$. The associativity of the star product $\ast$ is equivalent to the fact that $[L_f,R_g]=0$ for all $f,g \in C^\infty(M)[[\nu]]$. A star-product on a Poisson manifold $M$ can be restricted to any open subset of $M$.

We call a Poisson tensor on a complex manifiold $M$ a K\"ahler-Poisson tensor if it is of type (1,1) with respect to the complex structure. If a K\"ahler-Poisson tensor written in local coordinates $\{z^k, \bar z^l\}$ as $g^{lk}$ is nondegenerate, its inverse is a pseudo-K\"ahler metric tensor $g_{kl}$. We call a complex manifiold $M$ endowed with a K\"ahler-Poisson tensor a K\"ahler-Poisson manifold. Any pseudo-K\"ahler manifold is a K\"ahler-Poisson manifold. In this paper we will give several examples of K\"ahler-Poisson manifolds with the K\"ahler-Poisson tensor degenerate on the complement of an open dense subset. 

A star product (\ref{E:star}) on a K\"ahler-Poisson manifold defines a deformation quantization with separation of variables if the operators $C_r$ differentiate their first argument in antiholomorphic directions and the second argument in holomorphic ones. If the unit constant is the unity with respect to the star product, the condition of separation of variables can be equivalently stated as follows: for any local holomorphic function $a$ and a local antiholomorphic function $b$ the identities $a\ast f = af$ and $f\ast b = bf$ hold. Otherwise speaking, $L_a =a$ and $R_b = b$ are pointwise multiplication operators.

It is not known whether there exists a star product with separation of variables on an arbitrary K\"ahler-Poisson manifold. However, star products with separation of variables exist on any pseudo-K\"ahler manifold $M$ (see \cite{BW}, \cite{CMP1}).

Given a star product with separation of variables $\ast$ on a K\"ahler-Poisson manifold $M$, the formal Berezin transform of the star product $\ast$ is a formal differential operator $B = 1 + \nu B_1 + \nu^2 B_2 + \ldots$ globally defined on $M$ by the condition that 
\[
     B(ab) = b \ast a
\]
for any local holomorphic function $a$ and a local antiholomorphic function $b$. A star product with separation of variables can be recovered from its Berezin transform.

A deformation quantization with separation of variables on a pseudo-K\"ahler manifold $M$ equipped with a pseudo-K\"ahler form $\omega$ is called standard if its restriction to any contractible coordinate chart $(U, \{z^k\})$ has the property that
\[
   L_{\frac{\p \Phi}{\p z^k}} = \frac{\p \Phi}{\p z^k} + \nu\frac{\p }{\p z^k} \mbox{ and } 
R_{\frac{\p \Phi}{\p \bar z^l}} = \frac{\p \Phi}{\p \bar z^l} + \nu\frac{\p }{\p \bar z^l},
\]
where $\Phi$ is a potential of the pseudo-K\"ahler form $\omega$ on $U$, i.e., $\omega = i\p \bar \p \Phi$. This property defines the standard deformation quantization with separation of variables uniquely and globally on any pseudo-K\"ahler manifold $M$ (see \cite{CMP1}).

Let $M$ be a K\"ahler-Poisson manifold $M$ such that the K\"ahler-Poisson structure on $M$ given by a tensor $g^{lk}$ is nondegenerate on a dense open subset $\Omega$ of $M$ and its inverse on $\Omega$ is a metric tensor $g_{kl}$ with the corresponding pseudo-K\"ahler form $\omega$. It was shown in \cite{LMP} that the coefficients of the operators $C_1$ and $C_2$ of the standard star product with separation of variables are polynomials in partial derivatives of $g^{lk}$, while the operator $C_3$ is the sum of an operator with the same property and the operator
\[
   S(u,v) = g_{mn}\frac{\p g^{ls}}{\p \bar z^q}\frac{\p g^{np}}{\p z^s}\frac{\p g^{qm}}{\p \bar z^t}\frac{\p g^{tk}}{\p z^p}\frac{\p u}{\p \bar z^l}\frac{\p v}{\p z^k},   
\] 
which depends on the metric tensor $g_{kl}$. It follows from this observation that a star-product with separation of variables on $(\Omega,\omega)$ does not necessarily have a smooth extension to $M$. In this paper we give examples of K\"ahler-Poisson manifolds with open dense pseudo-K\"ahler submanifolds such that the standard deformation quantization with separation of variables on these pseudo-K\"ahler submanifolds admits a smooth extension to the whole manifold.

\medskip

\noindent{\bf Acknowledgments.} The author is very grateful to the participants of the conference ``Quantization of Singular Spaces" held at Aarhus University in December 2010 for inspiring discussions.

\section {Examples of deformation quantizations with separation of variables on K\"ahler-Poisson manifolds}\label{S:examples}

In this section we will give two examples of a K\"ahler-Poisson manifold $M$ with the K\"ahler-Poisson structure which is nondegenerate on a dense open subset $\Omega$ and such that the standard deformation quantization with separation of variables on $\Omega$ admits a smooth extension to $M$.

\medskip

\noindent{\bf Example 1.}

\medskip

Let $\psi$ be a defining function of a Levi-nondegenerate hypersurface $\Sigma$ in an open set $U \subset \C^n$. This means that $\psi$ is a smooth real function on $U$ with the zero set $\Sigma$ and such that the Monge-Amp{\`e}re matrix
\begin{equation}\label{E:monge}
   \Gamma=  \left( \begin{array}{cc}
\frac{\p^2 \psi}{\p z^k \p \bar z^l} & \frac{\p\psi}{\p z^k}\\
\frac{\p \psi}{\p \bar z^l} 
 & \psi \end{array} \right)
\end{equation}
is nondegenerate at the points of $\Sigma$. Shrinking, if necessary, the neighborhood $U$ around $\Sigma$, we may assume that the matrix $\Gamma$ is nondegenerate on $U$.  On the complement $U \backslash \Sigma$ of $\Sigma$ the potential $\log |\psi|$ defines a pseudo-K\"ahler form $\omega$ whose inverse is a K\"ahler-Poisson bivector field which has a smooth extension to $U$ by zero (i.e., vanishing on $\Sigma$). The following theorem was proved in \cite{CM}: 
\begin{theorem}\label{T:cm}
  The standard star product with separation of variables on $(U \backslash \Sigma,\omega)$ admits a smooth extension to a star-product on $U$.
\end{theorem}

Similar statements were proved earlier by different methods in \cite{E} and \cite{LTW}.

\medskip

\noindent{\bf Example 2.}

\medskip

The following example of a K\"ahler-Poisson manifold comes from the theory of complex symmetric domains. Denote by $\E$ the set of nondegenerate complex $(p+r)\times r$ matrices with the right action of the group $GL(r,\C)$. Then $Gr(r,p+r) = \E/GL(r,\C)$ is the Grassmannian of $r$-dimensional subspaces in $\C^{p+r}$. Consider the indefinite metric
\begin{equation}\label{E:indef}
    \sum_{k=1}^p z_k\bar w_k - \sum_{k = p+1}^{p+r}z_k\bar w_k
\end{equation}
on $\C^{p+r}$. The left action of the group $U(p,r)$ on $\E$ induces an action on the Grassmannian $Gr(r,p+r)$. Let $\Omega$ be the set of points of the Grassmannian $Gr(r,p+r)$ corresponding to the subspaces of $\C^{p+r}$ such that the restriction of the indefinite metric (\ref{E:indef}) to them is nondegenerate. It is a dense open $U(p,r)$-invariant subset of $Gr(r,p+r)$. Given a matrix $A\in \E$, denote by $U_A$ and $V_A$ its blocks of size $p\times r$ (top $p$ rows) and $r\times r$ (bottom $r$ rows), respectively. Denote by $\E'$ the set of matrices $A\in \E$ such that the block $V_A$ is nondegenerate. Then ${\mathcal Z}=\E' /GL(r,\C)$ is an affine subset of $Gr(r,p+r)$ parametrized by the $p\times r$-matrices $Z$ so that the coset of $A \in \E'$ corresponds to the matrix $Z = U_A/V_A$. The elements $\{z_{k\alpha}\}, 1 \leq k \leq p, 1\leq \alpha \leq r,$ of a $p\times r$-matrix $Z$ are thus holomorphic coordinates on ${\mathcal Z}$. The set ${\mathcal Z}\cap \Omega$ is parametrized by the matrices $Z$ such that $E-Z^\dagger Z$ is nondegenerate. The pseudo-K\"ahler metric on ${\mathcal Z}\cap \Omega$ defined by the potential 
\[
\Phi(Z,Z^\dagger) = \log|\det (E-Z^\dagger Z)|
\]
extends to a $U(p,r)$-invariant pseudo-K\"ahler metric on  $\Omega$. In particular, the set $\Omega_+ \subset {\mathcal Z}\subset Gr(r,p+r)$ parametrized by the matrices $Z$ such that $E-Z^\dagger Z$ is positive definite is an open $U(p,r)$-invariant subset of $Gr(r,p+r)$. It is a bounded symmetric domain and the metric corresponding to the potential $\Phi$ is K\"ahler. The standard deformation quantization with separation of variables on $\Omega$ is $U(p,r)$-invariant. The corresponding formal Berezin transform $B = 1 + \nu B_1 + \ldots$ is a $U(p,r)$-invariant formal differential operator on $\Omega$. It is known (see \cite{H}) that all $U(p,r)$-invariant differential operators on $\Omega_+$ are induced by the elements of the center of the universal enveloping algebra of the Lie algebra $u(p,r)$. These elements induce global $U(p,r)$-invariant operators on the Grassmannian $Gr(r,p+r)$. Since the coefficients of $U(p,r)$-invariant differential operators are algebraic functions on $Gr(r,p+r)$, it implies that the formal Berezin transform $B$ and the corresponding star product smoothly extend to $Gr(r,p+r)$. The corresponding $U(p,r)$-invariant K\"ahler-Poisson bivector is also globally defined on $Gr(r,p+r)$. In the coordinates $(Z,Z^\dagger)$ it is given by the formula
\[
      i\left(\delta_{st} - \sum_{\gamma=1}^r z_{s\gamma}\bar z_{t\gamma}\right)\left(\delta_{\alpha\beta} - \sum_{k=1}^p \bar z_{k\alpha} z_{k\beta}\right)\frac{\p}{\p z_{s\beta}}\wedge\frac{\p}{\p \bar z_{t\alpha}}.
\]
On $\Omega$ it is the inverse of the pseudo-K\"ahler form corresponding to the potential $\Phi$.

\section{Smooth extensions of star products}

Given an open subset $U$ of a smooth real $n$-dimensional manifold $M$, an $n$-tuple of smooth complex-valued functions $\{f^1,\ldots,f^n\}$ on $U$ is called a frame if for each point $x\in U$ the differentials $df^1(x), \ldots df^n(x)$ form a basis of the complexified cotangent space $T^\ast_xM\otimes \C$. An $n$-tuple of smooth formal complex-valued functions $f^k = f^k_0 + \nu f^k_1 + \nu^2 f^k_2 + \ldots, 1 \leq k \leq n,$ on $U$ is called a formal frame if $\{f^1_0, \dots, f^n_0\}$ is a frame on $U$.  

\begin{lemma}\label{L:ext}
 Let $U\subset \R^n$ be an open set with a dense open subset $V\subset U$ and $\{f^1,\ldots,f^n\}$ be a frame on $U$. If $A$ is a differential operator of finite order on $V$ such that the function $A1$ and the operators $[\ldots [[A, f^{k_1}],f^{k_2}],\ldots,f^{k_N}]$ have smooth extensions to $U$ for any $N$ and any indices $k_i, 1 \leq k_i \leq n$, then the operator $A$ has a smooth extension to $U$.
\end{lemma}
\begin{proof}
The lemma will be proved by induction on the order of the operator $A$. If $A$ is of order zero, it is the operator of pointwise multiplication by the function $A1$, which has a smooth extension to $U$. Assume that the statement of the lemma is true for any operator of order less than $r$ and that $A$ is of order $r$. Then for any indices $k_i, 1 \leq i \leq r$, the following identity holds:
\[
    [\ldots [[A,f^{k_1}],f^{k_2}],\ldots,f^{k_r}] = r! p\left(df^{k_1}\otimes\ldots\otimes df^{k_r}\right),
\]
where $p: (T^\ast V)^{\otimes r}\to \C$ is the (polarized) principal symbol of the operator $A$. Since the functions $\{f^i\}$ form a frame on $U$, the principal symbol of the operator $A$ has a smooth extension to $U$. One can construct an operator $B$ of order $r$ on $U$ whose principal symbol is the extension of $p$ to $U$. Now the operator $A-B$ is of order less than $r$ and satisfies the conditions of the lemma. Therefore $A-B$ has a smooth extension to $U$, whence the lemma follows.
\end{proof}

Let $M$ be a smooth real $n$-dimensional manifold with a dense open subset $\Omega$. Assume that $\pi$ is a Poisson bivector field on $\Omega$ and $\ast$ is a star product on the Poisson manifold $(\Omega,\pi)$.
\begin{theorem}\label{T:ext}
 Given a point $a \in M\backslash \Omega$ in a coordinate chart $U \subset M$, let $\{f^1,\ldots, f^n\}$ be a formal frame on $U$. If the operators of right star-multiplication $R_{f^k}, k = 1,\ldots,n$ (or the operators of left star multiplication $L_{f^k}, k = 1,\ldots,n$) on $\Omega\cap U$ can be extended to smooth formal differential operators on $U$, then the star-product $\ast$ has a smooth extension to a star-product on $U$. In particular, then $\pi$ extends to a smooth Poisson bivector field on $U$.
\end{theorem}
\begin{proof}
Let $u = u_0 + \nu u_1 +\ldots$ be a smooth formal function on $U$. The left star-multiplication operator $L_u$ on $U\cap \Omega$ commutes with the operators $R_{f^k}, k = 1,\ldots,n$.  Writing $L_u = A_0 + \nu A_1 + \dots$ and $R_{f^k} = B^k_0 + \nu B^k_1 + \ldots$, we have that the operators $A_0 = u_0$ and $B^k_r, r \geq 0$, have smooth extensions to $U$. We will prove by induction on $r$ that the operator $A_r$ has a smooth extension to $U$. This is true for $r=0$. Assume that this is true for all $r < s$. We have
\begin{equation}\label{E:ABB}
      [\ldots [[L_u,R_{f^{k_1}}],R_{f^{k_2}}], \ldots, R_{f^{k_N}}] = 0
\end{equation}
for any $N$ and indices $k_i$. Consider the coefficient at $\nu^s$ of the left-hand side of (\ref{E:ABB}). Since $B^k_0 = f^k_0$, this coefficient can be written as
\begin{equation}\label{E:Aff}
    [\ldots [[A_s,f^{k_1}_0],f^{k_2}_0], \ldots ,f^{k_N}_0]
\end{equation}
plus a sum of commutators of the operators $A_i$ and $B^k_j$ with $i < s$ and $j \leq s$ which all have smooth extensions to $U$. Thus the operator (\ref{E:Aff}) also has a smooth extension to $U$. Taking into account that $A_s1 = u_s$, we get from Lemma \ref{L:ext} that the operator $A_s$ has a smooth extension to $U$. Therefore the operator $L_u$ has a smooth extension to $U$ for any formal function $u$ on $U$. This implies that the star product $\ast$ extends to a smooth formal star-product on $U$. In particular, $\pi$ extends to a Poisson bivector field on $U$. 
\end{proof}

\section{A K\"ahler-Poisson tensor vanishing on a Levi-nondegenerate hypersurface}\label{S:niszero}

In this section we want to give yet another proof of Theorem~\ref{T:cm} from Example 1 based upon Theorem~\ref{T:ext}.

Recall that $\psi$ is a defining function of a Levi-nondegenerate hypersurface $\Sigma$ in an open set $U \subset \C^n$ and $\Gamma$ is the  Monge-Amp{\`e}re matrix (\ref{E:monge}) which we assume to be nondegenerate on $U$. Fix a point $x_0 \in \Sigma$. Then $\psi(x_0)=0$ and $(\p \psi)(x_0) \neq 0$, since the matrix $\Gamma(x_0)$ is nondegenerate. Therefore, there exists an index $s$ such that $\frac{\p \psi}{\p z^s}(x_0) \neq 0$. Denote by $V$ the neighborhood of $x_0$ within $U$ where $\frac{\p \psi}{\p z^s}$ does not vanish. We will construct formal functions $\{f^1,\ldots, f^n\}$ on $V$ such that the functions $\{z^1, \ldots, z^n, f^1,\ldots, f^n\}$ form a formal frame on a neighborhood $W \subset V$ of $x_0$ and the  operators $L_{f^k}, k = 1,\ldots,n,$ of the star product $\ast$ on $V\backslash \Sigma$ have a smooth extension to $W$ (this is trivially true for the operators $L_{z^k} = z^k$). Theorem \ref{T:ext} will then imply that the standard star product with separation of variables $\ast$ on $(U\backslash \Sigma,\omega)$ smoothly extends from $U\backslash \Sigma$ to $U$. Introduce the following invertible operator on $V$,
\[
    Q = 1 + \nu\psi \left(\frac{\p\psi}{\p z^s}\right)^{-1} \frac{\p}{\p z^s}.
\]

On $V\backslash \Sigma$ the operator
\begin{align*}
    L_{\frac{\p \log |\psi|}{\p z^s}} = \frac{\p \log |\psi|}{\p z^s} + \nu \frac{\p}{\p z^s} = \psi^{-1}\frac{\p \psi}{\p z^s} + \nu \frac{\p}{\p z^s} =  \psi^{-1}\frac{\p \psi}{\p z^s} Q
\end{align*}
is invertible. The inverse operator 
\[
   \left(L_{\frac{\p \log |\psi|}{\p z^s}}\right)^{-1} = Q^{-1}\circ \left(\left(\frac{\p\psi}{\p z^s}\right)^{-1}\psi\right)
\]
is also a left multiplication operator of the star product $\ast$ on $V\backslash \Sigma$. It admits a smooth extension to $V$ which we will denote $X^s$. Then $f^s = X^s 1$ is a smooth formal function on $V$, $f^s = f^s_0 + \nu f^s_1 + \ldots$, such that $f^s \ast \frac{\p \log |\psi|}{\p z^s} =1$ on $V\backslash \Sigma$ and
\[
    f^s_0 = \left(\frac{\p\psi}{\p z^s}\right)^{-1}\psi
\]
on $V$. For $k \neq s$ the operator
\begin{align*}
   \left(L_{\frac{\p \log |\psi|}{\p z^s}}\right)^{-1} L_{\frac{\p \log |\psi|}{\p z^k}} = 
   Q^{-1} \circ \left(\left(\frac{\p\psi}{\p z^s}\right)^{-1}\left(\frac{\p \psi}{\p z^k} + \nu \psi\frac{\p}{\p z^k}\right)\right)
\end{align*}
is a left multiplication operator of the star product $\ast$ on $V\backslash \Sigma$. It admits a smooth extension to $V$ which we will denote $X^k$. Then $f^k = X^k 1$ is a smooth formal function on $V$, $f^k = f^k_0 + \nu f^k_1 + \ldots$, and
\[
      f^k_0 = \left(\frac{\p\psi}{\p z^s}\right)^{-1}\left(\frac{\p \psi}{\p z^k}\right).
\]
We want to prove that $\{z^1, \ldots, z^n, f^1,\ldots, f^n\}$ is a formal frame on a neighborhood of $x_0$. It suffices to prove that the covectors 
\begin{equation}\label{E:covector}
   \left(\frac{\p f^k_0}{\p\bar z^1}(x_0), \ldots,\frac{\p f^k_0}{\p\bar z^n}(x_0)\right)
\end{equation}
for $k = 1, \ldots,n,$ are linearly independent. Taking into account that $\psi(x_0)=0$, we see that for $k=s$ the covector (\ref{E:covector}) is nonzero and propor\-tional to the nonzero covector
\begin{equation}\label{E:s}
  \left(\frac{\p\psi}{\p\bar z^1}(x_0), \ldots,\frac{\p\psi}{\p\bar z^n}(x_0)\right).
\end{equation}
For $k \neq s$ the covector (\ref{E:covector}) is proportional to the covector
\begin{equation}\label{E:nots}
    \left(\frac{\p^2 \psi}{\p z^k \p\bar z^1}\frac{\p\psi}{\p z^s} - \frac{\p^2 \psi}{\p z^s \p\bar z^1}\frac{\p\psi}{\p z^k}, \ldots, \frac{\p^2 \psi}{\p z^k \p\bar z^n}\frac{\p\psi}{\p z^s} - \frac{\p^2 \psi}{\p z^s \p\bar z^n}\frac{\p\psi}{\p z^k} \right)
\end{equation}
at $x_0$. It is a simple consequence of formulas (\ref{E:s}) and (\ref{E:nots}) that the linear independence of the covectors (\ref{E:covector}) for $k = 1, \ldots,n,$ is equivalent to the nondegeneracy of the matrix $\Gamma(x_0)$.

Thus we have proved Theorem \ref{T:cm} from \cite{CM} using a different approach.

\section{Roots of formal differential operators}\label{S:roots}

In the next section we will construct a family of star products with separation of variables on the complement of a Levi-nondegenerate hypersurface in an open subset of $\C^n$ which admit a smooth extension to the whole open subset. In order to use Theorem \ref{T:ext} we will have to find a root of a specific formal differential operator. To this end we will now prove several technical statements.

Let ${\mathcal X}$ denote the ring of polynomials in an infinite number of variables,
\[
    {\mathcal X} = \C [t_0,t_1,\ldots],
\]
and $\A$ denote the algebra of differential operators on ${\mathcal X}$ generated by the multiplication operators by the elements of ${\mathcal X}$ and a single differentiation operator
\[
    \delta = \sum_{k=0}^\infty t_{k+1} \frac{\p}{\p t_k}.
\]
For this operator, $\delta t_k = t_{k+1}$. Clearly, $\A$ is generated by $t_0$ and $\delta$.
Given a manifold $X$, a function $f$, and a vector field $v$ on $X$, denote by $D$ the algebra of differential operators on $X$ generated by $f$ and $v$. Then there exists a well defined surjective homomorphism $\tau: \A \to D$ such that $\tau(t_0) = f$ and $\tau(\delta) = v$.

\begin{lemma}\label{L:divisionina}
 Let $B_r$ be a differential operator in $\A$ of order not greater than $r$. Then for any natural number $N$ there exists a unique differential operator $A_r \in \A$  satisfying the equation
\begin{equation}\label{E:mainina}
     \sum_{i =0}^N  t_0^{N-i} A_r \circ t_0^i = t_0^{N(r+1)}B_r.
\end{equation}
The order of $A_r$ is not greater than $r$.
\end{lemma}
\begin{proof}
 We will prove that equation (\ref{E:mainina}) has a unique solution by induction on $r$. Comparing the principal symbols of  the operators on both sides of equation (\ref{E:mainina}) we see that the order of the operator $A_r$ must be equal to the order of $B_r$.  First consider equation (\ref{E:mainina}) with $r=0$. Both $A_0$ and $B_0$ must be multiplication operators by elements of ${\mathcal X}$ and
\[
    A_0 = \frac{1}{N+1}B_0.
\]

Denote the principal symbol of order $p$ of a differential operator $X$ of order not greater than $p$ by $\sigma_p(X)$. If $A_r$ is a solution of equation (\ref{E:mainina}), then
\[
      (N+1) t_0^N \sigma_r(A_r) = t_0^{N(r+1)}\sigma_r(B_r),
\]
which implies that
\[
      \sigma_r(A_r) = \frac{1}{N+1} t_0^{Nr}\sigma_r(B_r).
\]
Therefore the order of the operator 
\[
A_{r-1}: = A_r - \frac{1}{N+1} t_0^{Nr} B_r
\]
must be not greater than $r-1$ and $A_{r-1}$ should satisfy equation (\ref{E:mainina}) with $r$ replaced with $r-1$ and with
\[
     B_{r-1} := t_0^N B_r - \frac{1}{N+1} \sum_{i =0}^N t_0^{N-i} B_r \circ t_0^i.
\] 
It is clear that $\sigma_r(B_{r-1})=0$, whence the order of $B_{r-1}$ is not greater than 
$r -1$. By the induction principle, it implies that equation (\ref{E:mainina}) has a unique solution for any $r$.
\end{proof}

We introduce a bidegree on the algebra $\A$ such that the operator $t_0$ has the bidegree $(1,0)$ and the operator $\delta$ has the bidegree $(0,1)$. Then the operator $t_r$ has the bidegree $(1,r)$. Observe that if the operator $B_r$ from Lemma \ref{L:divisionina} is a homogeneous element of algebra $\A$ of bidegree $(q,r)$, then the solution $A_r$ of eqn. (\ref{E:mainina}) is homogeneous of bidegree $(Nr + q,r)$.  

\begin{lemma}\label{L:root}
Given the operator
\begin{equation}\label{E:S}
    S: = \sum_{k=0}^\infty \nu^k \left(t_0^{N+1} \delta\right)^k \circ  t_0^{N+1} = t_0^{N+1} + \nu t_0^{N+1} \delta \circ t_0^{N+1} + \ldots
\end{equation}
in the algebra $\A[[\nu]]$, there exists a unique operator  $A = A_0 + \nu A_1 + \ldots$ in $\A[[\nu]]$ such that $A_0 = t_0$ and 
\begin{equation}\label{E:power}
A^{N+1} = S. 
\end{equation}
\end{lemma}
\begin{proof}
 Equating the coefficients at $\nu^r$ of the operators on the both sides of equation (\ref{E:power}), we obtain the equation
\[
    \sum_{i_0 + \ldots i_N = r} A_{i_0}\ldots A_{i_N} = \left( t_0^{N+1}\delta\right)^r \circ t_0^{N+1} 
\]
which can be rewritten as follows:
\begin{equation}\label{E:ind}
 \sum_{i =0}^N t_0^{N-i} A_r \circ t_0^i =   \left( t_0^{N+1}\delta\right)^r \circ t_0^{N+1} - \sum_{i_0 + \ldots i_N = r, i_s < r} A_{i_0}\ldots A_{i_N}. 
\end{equation}
Notice that the right-hand side of eqn. (\ref{E:ind}) depends only on the operators $A_i$ with $i <r$.  We can find the components $A_r$ from equation (\ref{E:ind}) by induction on $r$ using Lemma \ref{L:divisionina}. Applying induction to eqn. (\ref{E:ind}), we have to show simultaneously that the bidegree of $A_r$ is $((N+1)r + 1,r)$ and that the right-hand side of  eqn. (\ref{E:ind}) can be represented in the form $t_0^{N(r+1)}B_r$ for some operator $B_r \in \A$. To justify the latter statement observe that the right-hand side of eqn. (\ref{E:ind}) being a homogeneous element of the algebra $\A$ of bidegree $((N+1)(r+1),r)$, can be written as a linear combination of operators of the form
\[
   \left(\prod_{s \geq 0}\left(t_s\right)^{i_s}\right)\delta^j, 
\]
where 
\[
  \sum_{s \geq 0} i_s = (N+1)(r+1) \mbox{ and } \sum_{s \geq 1} s i_s  + j = r.
\]
It implies that
\[
    i_0 = \left(\sum_{s \geq 2} (s - 1) i_s\right) + j + Nr + N + 1  > N(r+1),
\]
which means that any homogeneous element of the algebra $\A$ of bidegree $((N+1)(r+1),r)$ is divisible on the left by $t_0^{N(r+1)}$, which concludes the proof.
\end{proof}

We will also need the two following  lemmas.

\begin{lemma}\label{L:equalroots}
Given a nonvanishing smooth function $f$ and two formal differential operators $A = A_0 + \nu A_1 + \ldots$ and $B = B_0 + \nu B_1 + \ldots$ on a manifold $M$ such that $A_0=B_0=f$ is the pointwise multiplication operator by $f$ and $A^{N+1} = B^{N+1}$ for a nonnegative integer $N$, then $A = B$.
\end{lemma}
\begin{proof}
Let $D$ be a differential operator  on $M$ such that 
\begin{equation}\label{E:zero}
f^N D + f^{N-1} D \circ f + \ldots + D \circ f^N =0. 
\end{equation}
Assume that $D$ is a nonzero operator of order $r$ with the nonzero principal symbol $\sigma_r(D)$. Now, the principal symbol of the left-hand side of (\ref{E:zero}) is $(N+1) f^N \sigma_r(D)=0$, whence $\sigma_r(D)=0$. This contradiction implies that $D=0$. Using this observation, we will prove by induction on $r$ that $A_r = B_r$ for all $r \geq 0$. We have that $A_0=B_0$. Given $r>0$, assume that $A_k=B_k$ for all $k <r$. Denote 
\[
X := A_0 + \nu A_1 + \ldots + \nu^{r-1} A_{r-1} = B_0 + \nu B_1 + \ldots + \nu^{r-1} B_{r-1} . 
\]
It follows from the condition $A^{N+1} = B^{N+1}$ that 
\begin{eqnarray*}
 \nu^r \left(X^N A_r + X^{N-1} A_r X + \ldots A_r X^N \right) = & \\
 \nu^r \left(X^N B_r + X^{N-1} B_r X + \ldots B_r X^N \right) &  \pmod{\nu^{r+1}}, 
\end{eqnarray*}
whence  $f^N D + f^{N-1} D \circ f + \ldots + D \circ f^N =0$ for $D = A_r - B_r$. Therefore, $A_r = B_r$, which concludes the proof. 
\end{proof}

\begin{lemma}\label{L:roots}
 Given a star product $\ast$ on a Poison manifold $M$, a nonvanishing complex-valued function $u_0$ on $M$, and a formal function $v = v_0 + \nu v_1 + \ldots$ such that $v_0 = u_0^q$ for some $q \in \Z$, there exists a unique formal function $u = u_0 + \nu u_1 + \ldots$ on $M$ such that $v = u^{\ast q}$.
\end{lemma}
Here $u^{\ast q}$ is the $q$th power of $u$ with respect to the star product $\ast$.
\begin{proof}
 Assume that $q>0$. We will show by induction on $l$ the existence and uniqueness of each coefficient $u_l, l > 1$. For $f_1,\ldots f_q \in C^\infty(M)$ set
\[
      f_1 \ast \ldots \ast f_q = \sum_{r=0}^\infty \nu^r C_r(f_1, \ldots, f_q).
\]
Then, in particular,  $C_0(f_1, \ldots, f_q) = f_1 \ldots f_q .$
Equating the coefficients at $\nu^l$ of $u^{\ast q}$ and $v$ we get
\begin{equation}\label{E:qfold}
   \sum_{i_0 + \ldots + i_q = l} C_{i_0}(u_{i_1}, \ldots, u_{i_q}) = v_l.
\end{equation}
The terms containing $u_l$ on the left hand side of (\ref{E:qfold}) are
\[
   C_0(u_l, u_0,\ldots u_0) + C_0(u_0, u_l,\ldots u_0) + \ldots + C_0(u_0,\ldots, u_l) = q 
u_0^{q-1} u_l,
\]
which shows that $u_l$ is uniquely expressed through the coefficients $u_j$ for $j < l$.
The statement of the lemma is well known for $q = -1$. Assume that $q<0$.  The equation $v = u^{\ast q}$ is equivalent to $v^{\ast(-1)} = u^{\ast(-q)}$ which reduces the case of $q <0$ to the  case of $q >0$.
\end{proof}

\section{A family of  K\"ahler-Poisson tensors vanishing on a Levi-nondegenerate hypersurface}

Given an open subset $U \subset \C^n$ and a Levi-nondegenerate hypersurface $\Sigma \subset U$ with a defining function $\psi$, for each positive integer $N$ we will introduce a K\"ahler-Poisson tensor on a neighborhood $U_N$ of $\Sigma$ in $U$ vanishing on $\Sigma$ and nondegenerate on its complement $U_N \backslash \Sigma$ such that the corresponding standard deformation quantization on $U_N \backslash \Sigma$ admits a smooth extension to $U_N$.

For each nonnegative integer $N$ define a matrix
\[
   \Gamma_N=  \left( \begin{array}{cc}
\frac{\p^2 \psi}{\p z^k \p \bar z^l} & \frac{\p\psi}{\p z^k}\\
\frac{\p \psi}{\p \bar z^l} 
 & (N+1)^{-1} \psi \end{array} \right)
\]
on $U$. In particular, $\Gamma_0 = \Gamma$. Set $\Omega = U \backslash \Sigma$. For $N >0$, let $\omega_N$ be the closed $(1,1)$-form on $\Omega$ whose potential is $\Phi = \frac{1-\psi^{-N}}{N}$. Set
\[
 g_{kl} = \frac{\p^2 \Phi}{\p z^k \p \bar z^l}.
\]
\begin{lemma}\label{L:gamman}
 The form  $\omega_N$ is nondegenerate if and only if the matrix $\Gamma_N$ is nondegenerate.
\end{lemma}
\begin{proof}
 For each $k \leq n$ multiply the last row in the matrix $\Gamma_N$ by $(N+1)\psi^{-1}\frac{\p\psi}{\p z^k}$ and subtract it from the $k$th row. The resulting matrix is
\[
\left( \begin{array}{cc}
 X_{kl} & 0\\
\frac{\p \psi}{\p \bar z^l} 
 & (N+1)^{-1} \psi \end{array} \right),
\]
where $X_{kl} = \frac{\p^2 \psi}{\p z^k \p \bar z^l} - (N+1)\psi^{-1}\frac{\p\psi}{\p z^k} \frac{\p \psi}{\p \bar z^l}$. The lemma follows from the observation that
\begin{equation}\label{E:niszero}
     g_{kl} = \psi^{-N-1}X_{kl}.
\end{equation}
\end{proof}

{\it Remark.} If $N=0$, the tensor (\ref{E:niszero}) defines the $(1,1)$-form $\omega$ with the potential $\log |\psi|$ as in Example 1.

Since $\psi$ is a defining function of the Levi-nondegenerate hypersurface $\Sigma \subset U$, the matrix $\Gamma_N$ is nondegenerate at every point of $\Sigma$. Thus, the matrix $\Gamma_N$ is nondegenerate on some neighborhood $U_N$ of $\Sigma$ in $U$
and therefore $\omega_N$ is a pseudo-K\"ahler form on $U_N \backslash \Sigma$.

Consider the inverse matrix
\[
   \Gamma_N^{-1}=  \left( \begin{array}{cc}
A^{lm} & B^l\\
C^m & D \end{array} \right)
\]
on $U_N$. A simple calculation shows that the matrix $A^{lm}$ is inverse to $X_{kl}$, which implies that the inverse $g^{lm}$ of the matrix $g_{kl}$ is 
\[
    g^{lm} = \psi^{N+1}A^{lm}. 
\]
Taking into account that the matrix $A^{lm}$ is smooth on $U_N$, we see that the K\"ahler-Poisson tensor $g^{lm}$ admits a smooth extension to $U_N$ which vanishes on $U_N \cap \Sigma$. We will prove that the standard star-product with separation of variables $\ast$ on $(U_N \backslash \Sigma, \omega_N)$ also admits a smooth extension to $U_N$. 

As in Section \ref{S:niszero}, assume that $x_0$ is an arbitrary point in $\Sigma$ and $s$ is an index such that $\frac{\p \psi}{\p z^s}\neq 0$ on some neighborhood $V \subset U_N$ of $x_0$. On $U_N \backslash \Sigma$ we have
\begin{align*}
   L_{\frac{\p \Phi}{\p z^k}} =  \frac{\p \Phi}{\p z^k} + \nu \frac{\p}{\p z^k} & = \psi^{-N-1}\frac{\p \psi}{\p z^k} + \nu \frac{\p}{\p z^k} =\\
& \psi^{-N-1}\left( \frac{\p \psi}{\p z^k} + \nu \psi^{N+1}\frac{\p}{\p z^k}\right)
\end{align*}
and the operator $L_{\frac{\p \Phi}{\p z^s}}$ is invertible on $V \backslash \Sigma$. Moreover, its inverse
\begin{equation}\label{E:opery}
    \left(L_{\frac{\p \Phi}{\p z^s}}\right)^{-1} = \left( 1 + \nu \psi^{N+1}\left(\frac{\p \psi}{\p z^s}\right)^{-1}\frac{\p}{\p z^s}\right)^{-1} \circ \left(\left(\frac{\p \psi}{\p z^s}\right)^{-1}\psi^{N+1}\right)
\end{equation}
and the operators
\begin{align}\label{E:opers}
  & \left(L_{\frac{\p \Phi}{\p z^s}}\right)^{-1}L_{\frac{\p \Phi}{\p z^k}} = \nonumber
\\ 
 & \left( 1 + \nu \psi^{N+1}  \left(\frac{\p \psi}{\p z^s}\right)^{-1}\frac{\p}{\p z^s}\right)^{-1} \circ \left(\left(\frac{\p \psi}{\p z^s}\right)^{-1}\left( \frac{\p \psi}{\p z^k} + \nu \psi^{N+1}\frac{\p}{\p z^k}\right)\right) 
\end{align}
for $k \neq s$ admit smooth extensions to $V$.

As in Section \ref{S:niszero}, we want to construct a formal frame  
\[
\{z^1, \ldots, z^n, f^1,\ldots, f^n\}
\]  
on some neighborhood $W \subset V$ of $x_0$ such that the operators $L_{f^k}, k = 1,\ldots,n,$ of the standard star product with separation of variables on $W\backslash \Sigma$ have smooth extensions to $W$. 
For $k \neq s$ denote by $X_k$ the smooth extension of the operator (\ref{E:opers}) to $V$
and set $f^k = X_k1$. Then the coefficient $f^k_0$ at the zero degree of the formal parameter $\nu$ in $f^k$ is given by the formula
\[
      f^k_0 = \left(\frac{\p\psi}{\p z^s}\right)^{-1}\left(\frac{\p \psi}{\p z^k}\right).
\]
as in Section \ref{S:niszero}.  To define the function $f^s$ in the formal frame we will show that the smooth extension of the operator (\ref {E:opery}) to $V$ has a smooth root of degree $N+1$ on a neighborhood of $x_0$. Shrinking, if necessary, the neighborhood $V$ of $x_0$, denote by $\chi$ any root of degree $N+1$ of the function $\left(\frac{\p \psi}{\p z^s}\right)^{-1}$ on $V$. According to Lemma ~\ref{L:roots}, there exists a unique formal function $u = u_0 + \nu u_1 + \ldots$ on $V\backslash \Sigma$ such that $u_0 = \chi\psi$ and  $u^{\ast(-N-1)} = \psi^{-N-1} \frac{\p \psi}{\p z^s} = \frac{\p \Phi}{\p z^s}$. Therefore,
\[
   \left(L_u \right)^{N+1} = \left(L_{\frac{\p \Phi}{\p z^s}}\right)^{-1} 
\]
on $V \backslash \Sigma$. On the other hand, the operator $L_{\frac{\p \Phi}{\p z^s}}$ can be written in the form
\[
    L_{\frac{\p \Phi}{\p z^s}} =  \psi^{-N-1}\frac{\p \psi}{\p z^s} + \nu \frac{\p}{\p z^s} = 
(\psi\chi)^{-N-1} + \nu \frac{\p}{\p z^s}
\]
on $V\backslash\Sigma$, and its inverse,
\[
   \left(L_{\frac{\p \Phi}{\p z^s}}\right)^{-1} = \sum_{k=0}^\infty \nu^k \left( (\psi\chi)^{N+1} 
\left(-\frac{\p}{\p z^s}\right) \right)^k \circ  (\psi\chi)^{N+1},
\]
has a smooth extension to $V$.
Consider a homomorphism $\tau$ from the algebra $\A$ introduced in Section \ref{S:roots} to the algebra of differential operators on $V$ such that $\tau(t_0) = \psi\chi$ and $\tau(\delta) = -\frac{\p}{\p z^s}$. Extend it to the mapping from $\A[[\nu]]$ to the algebra of formal differential operators on $V$ by $\nu$-linearity. Then
\[
     \tau(S)  = \left(L_{\frac{\p \Phi}{\p z^s}}\right)^{-1} = \left(L_u \right)^{N+1}
\]
for $S\in \A[[\nu]]$ given by formula (\ref{E:S}). According to Lemma \ref{L:root}, there exists an element $A \in \A[[\nu]]$ such that $A^{N+1} = S$ and $A = t_0 \pmod{\nu}$. Therefore,
$\tau(A) = \psi\chi \pmod{\nu}$ and
\[
     \left(\tau(A)\right)^{N+1} = \left(L_{\frac{\p \Phi}{\p z^s}}\right)^{-1} = (L_u)^{N+1}
\]
on $V \backslash\Sigma$. It follows from Lemma \ref{L:equalroots} that $\tau(A)=L_u$ on $V \backslash\Sigma$. Therefore the operator $\tau(A)$ is a smooth extension of the operator $L_u$ to $V$. Set $f^s = \tau(A)1$. The coefficient of $f^s$ at the zero degree of $\nu$ is $f^s_0 = \psi \chi$. Now, taking into account that $\psi (x_0)=0$, we get
\[
    \frac{\p f^s_0}{\p \bar z^l}(x_0) = \left( \frac{\p \psi}{\p \bar z^l}\,  \chi\right)(x_0).
\]
Since $ \chi(x_0) \neq 0$, we obtain that the covector $\bar \p f^s_0 (x_0)$ is proportional to the nonzero covector $\bar \p \psi (x_0)$ and is nonzero itself. Thus it can be proved as in Section \ref{S:niszero} that the formal functions
\begin{equation}\label{E:formfram}
\{z^1, \ldots, z^n, f^1,\ldots, f^n\}
\end{equation}
form a formal frame on a neighborhood of the point $x_0$. This formal frame satisfies the conditions of Theorem \ref{T:ext}, which concludes the proof of the following theorem:
\begin{theorem}
  Let $\psi$ be a defining function of a Levi-nondegenerate hypersurface $\Sigma$ in an open set $U \subset \C^n$ and $N$ be a natural number. Then there exists a neighborhood $U_N \subset U$ of $\Sigma$ such that the potential
\[
     \frac{1 - \psi^{-N}}{N}
\]
defines a pseudo-K\"ahler form $\omega_N$ on $U_N \backslash \Sigma$, and both the corresponding K\"ahler-Poisson structure and the standard deformation quantization with separation of variables on $(U_N \backslash \Sigma,\omega_N)$ admit smooth extensions to the neighborhood $U_N$.
\end{theorem}

\section{A smooth extension of a star product on an open subset of a Grassmannian}

In this section we will use Theorem \ref{T:ext} to prove that the star product with separation of variables from Example 2 in Section \ref{S:examples} admits a smooth extension. Namely, let $M$ be the set of complex $p\times r$-matrices $Z =(z_{k\alpha})$ and ${\mathcal O}$ be the subset of matrices $Z\in M$ such that the matrix $\Psi=E-Z^\dagger Z$ is nondegenerate. The potential $\Phi = \log |\det \Psi|$ determines a pseudo-K\"ahler form $\omega$  on ${\mathcal O}$. 

\begin{theorem}
 The K\"ahler-Poisson structure corresponding to the pseudo-K\"ahler structure on ${\mathcal O}$ given by the form $\omega$ and the standard star product with separation of variables $\ast$ on $({\mathcal O},\omega)$ admit smooth extensions to $M$.
\end{theorem}
\begin{proof}
 The matrix $\Psi=E-Z^\dagger Z$ has the entries
\[
\psi_{\alpha\beta} = \delta_{\alpha\beta} - \sum_{k=1}^p \bar z_{k\alpha} z_{k\beta}
\]
and is invertible on ${\mathcal O}$. Denote the inverse matrix by $X = (\chi_{\beta\gamma})$. Then
\[
   \frac{\p \Phi}{\p z_{k\varkappa}} =\frac{\p}{\p z_{k\varkappa}} \log|\det \Psi| = \sum_{\alpha,\beta}\chi_{\beta\alpha}\frac{\p \psi_{\alpha\beta}}{\p z_{k\varkappa}} = - \sum_{\alpha}\chi_{\varkappa\alpha}\bar z_{k\alpha}.
\] 
Therefore, the left multiplication operator by $u_{\varkappa k}:=\sum_{\alpha}\chi_{\varkappa\alpha}\bar z_{k\alpha}$ with respect to the star product $\ast$ is
\[
    L_{u_{\varkappa k}} = u_{\varkappa k} - \nu\frac{\p}{\p z_{k\varkappa}} = \sum_\alpha \chi_{\varkappa\alpha} \left(\bar z_{k\alpha} - \nu \sum_\gamma \psi_{\alpha\gamma}\frac{\p}{\p z_{k\gamma}} \right).
\]
Taking into account that 
\[
   \sum_k u_{\varkappa k}z_{k\beta} = \sum_{k,\alpha} \chi_{\varkappa\alpha}\bar z_{k\alpha}z_{k\beta} = 
\sum_\alpha \chi_{\varkappa\alpha}(\delta_{\alpha\beta} - \psi_{\alpha\beta}) = \chi_{\varkappa\beta} - \delta_{\varkappa\beta},
\]
we obtain that
\begin{align*}
  \sum_k L_{z_{k\beta}}L_{u_{\varkappa k}} = \sum_k \left(u_{\varkappa k}z_{k\beta} - \nu z_{k\beta}\frac{\p}{\p z_{k\varkappa}}\right) = \\
\chi_{\varkappa\beta} - \delta_{\varkappa\beta}  - \nu \sum_k z_{k\beta}\frac{\p}{\p z_{k\varkappa}}
\end{align*}
is the left multiplication operator by $\chi_{\varkappa\beta} - \delta_{\varkappa\beta}$. It follows that
\begin{align*}
    L_{\chi_{\varkappa\beta}} = \chi_{\varkappa\beta} - \nu \sum_k z_{k\beta}\frac{\p}{\p z_{k\varkappa}}=\\
 \sum_\alpha \chi_{\varkappa\alpha} \left(\delta_{\alpha\beta} - \nu \sum_{k,\lambda}\psi_{\alpha\lambda}z_{k\beta}\frac{\p}{\p z_{k\lambda}}\right).
\end{align*}
Interpreting the operators $L_{\chi_{\varkappa\beta}},\chi_{\varkappa\alpha}$ and
\begin{equation}\label{E:inv}
   \delta_{\alpha\beta} - \nu \sum_{k,\lambda}\psi_{\alpha\lambda}z_{k\beta}\frac{\p}{\p z_{k\lambda}}
\end{equation}
as matrices whose entries are formal differential operators on ${\mathcal O}$, we see that the matrix (\ref{E:inv}) has a smooth extension to $M$ and is invertible on $M$. Denote its inverse on $M$ by $Q_{\beta\alpha}$. The inverse matrix of $L_{\chi_{\varkappa\beta}}$ on ${\mathcal O}$ is
\[
    J_{\beta\varkappa}:=\sum_\alpha Q_{\beta\alpha}\circ\psi_{\alpha\varkappa}.
\]
Its entries $J_{\beta\varkappa}$ are left multiplication operators with respect to the product $\ast$.
The matrix $J_{\beta\varkappa}$ has a smooth extension to $M$. Observe that the entries of the matrix
\[
    K_{\beta k}: = \sum_\varkappa J_{\beta\varkappa} L_{u_{\varkappa k}} = \sum_\alpha Q_{\beta\alpha}\circ\left(\bar z_{k\alpha} - \nu \sum_\gamma \psi_{\alpha\gamma}\frac{\p}{\p z_{k\gamma}} \right)
\]
are also left multiplication operators with respect to the product $\ast$ and the matrix $K_{\beta k}$
also has a smooth extension to $M$. Denote by $f^{\beta k} = f^{\beta k}_0 + \nu f^{\beta k}_1 + \ldots$ the formal function on $M$ given by the formula $f^{\beta k} = K_{\beta k}1$. Since $f^{\beta k}_0 = \bar z_{k\beta}$, the functions $\{z_{j\alpha}, f^{\beta k}\}$ form a formal frame on $M$. Also, $L_{f^{\beta k}} = K_{\beta k}$ on ${\mathcal O}$. Now, it follows from Theorem \ref{T:ext} that the star product $\ast$ admits a smooth extension to $M$.
\end{proof}

\end{document}